\documentclass[a4paper,12pt,oneside]{article}
\usepackage[T1]{fontenc}
\usepackage[latin9]{inputenc}
\usepackage[scale=.8]{geometry}
\usepackage{ulem}
\usepackage{textcomp}
\usepackage{amsmath}
\usepackage{esint}
\usepackage[colorlinks,citecolor=blue,urlcolor=blue]{hyperref}
\usepackage{amssymb,amsthm}
\usepackage{esint}
\usepackage[colorlinks,citecolor=blue,urlcolor=blue]{hyperref}
\usepackage{graphicx,epsfig,color}

\makeatletter

\newcommand{\ddr}{\mathrm{d}}
\newcommand{\R}{\mathbb{R}}
\newcommand{\E}{\mathbb{E}}
\renewcommand{\P}{\mathbb{P}}

\newcommand{\un}{\mathbf{1}}
\newcommand{\noi}{\noindent}
\newcommand{\onu}{\overline{\nu}}
\newcommand{\bpsi}{\bar{\Psi}}

\newtheorem{thm}{Theorem}
\newtheorem{lem}[thm]{Lemma}
\newtheorem{cor}[thm]{Corollary}
\newtheorem{prop}[thm]{Proposition}
\newtheorem{rem}{Remark}[section]
\newtheorem{Example}{Example}[section]
\newtheorem*{Def}{Definition}

 \newtheorem{condition}[thm]{Condition}

 \def\blemma{\begin{lem}}\def\elemma{\end{lem}}
 \def\bproposition{\begin{prop}}\def\eproposition{\end{prop}}
 \def\btheorem{\begin{theorem}}\def\etheorem{\end{theorem}}
 \def\bcorollary{\begin{cor}}\def\ecorollary{\end{cor}}
 \def\bremark{\begin{rem}}\def\eremark{\end{rem}}
 \def\bcondition{\begin{condition}}\def\econdition{\end{condition}}

 \def\benumerate{\begin{enumerate}}\def\eenumerate{\end{enumerate}}
 \def\bitemize{\begin{itemize}}\def\eitemize{\end{itemize}}

 \def\beqlb{\begin{eqnarray}}\def\eeqlb{\end{eqnarray}}
 \def\beqnn{\begin{eqnarray*}}\def\eeqnn{\end{eqnarray*}}

\date{\today}
\title{\bf On the hitting times of continuous-state branching processes with immigration}
\begin{document}
\maketitle{}
\begin{center} {Xan Duhalde \footnotemark[1], Clément Foucart \footnotemark[2], Chunhua Ma \footnotemark[3]}
\bigskip
\footnotetext[1]{xan.duhalde@upmc.fr, Université Paris 6, Laboratoire de Probabilit\'es et Mod\`eles Al\'eatoires,  4 Place Jussieu- 75252 Paris Cedex 05- FRANCE.}
\footnotetext[2]{foucart@math.univ-paris13.fr, Université Paris 13,  Laboratoire Analyse, Géométrie \& Applications UMR 7539 Institut Galilée   99 avenue J.B. Clément
93430 Villetaneuse FRANCE.}
\footnotetext[3]{mach@nankai.edu.cn, Nankai University, School of Mathematical Sciences and LPMC,  Tianjin 300071 P.\,R.  CHINA.}
\end{center}
\begin{abstract}
We study the two-dimensional joint distribution of the first hitting time of a constant level by a continuous-state branching process with immigration and their primitive stopped at this time. We show an explicit expression of its Laplace transform and obtain a necessary and sufficient criterion for transience or recurrence. We follow the approach of Shiga, T. (1990) [A recurrence criterion for Markov processes of Ornstein-Uhlenbeck type. Probability Theory and Related Fields, 85(4), 425-447], by finding some $\lambda$-invariant functions for the generator.
\end{abstract}
 \vspace{9pt} \noindent {\bf Key words.}
{Continuous-state branching processes}, {immigration}, {first hitting time}, {transience and recurrence}, {polarity}.
\par \vspace{9pt}
\section{Introduction and main results.}
The continuous-state branching processes with immigration (CBI for short) are a class of time-homogeneous Markov process with values in $\mathbb{R}_{+}$. They have been introduced by Kawazu and Watanabe in 1971, see \cite{KAW}, as limits of rescaled Galton-Watson processes with immigration. They form an important class of Markov processes which has received significant attention in the literature. For an introduction to these processes, we refer to Li \cite{Lilecturenotes}, \cite{Li} and Kyprianou \cite{Kyprianoubook}.
\\

\noindent Any CBI process is characterized in law by a couple $(\Psi,\Phi)$ of Lévy-Khintchine functions :
\begin{align}
\Psi(q)&=\gamma q+\frac{1}{2}\sigma^{2}q^{2}+\int_{0}^{\infty}(e^{-qu}-1+qu1_{\{u\in (0,1)\}})\pi(\ddr u), \label{reproduction}\\
\Phi(q)&=bq+\int_{0}^{\infty}(1-e^{-qu})\nu(\ddr u) \label{immigration}
\end{align}
where $\sigma, b \geq 0$, $\gamma \in \mathbb{R}$ and $\nu$, $\pi$ are two L\'evy measures such that $\int_{0}^{\infty} (1\wedge u) \nu(\ddr u)<\infty$ and $\int_{0}^{\infty} (1\wedge u^{2})\pi(\ddr u)<\infty$. The measure $\pi$ is the L\'evy measure of a spectrally positive L\'evy process which characterizes the reproduction.  The measure $\nu$ characterizes the jumps of the subordinator that describes the arrival of immigrants in the population. The non-negative constants $\sigma$ and $b$ correspond respectively to the continuous reproduction and the continuous immigration.
To shorten our notation, a continuous-state branching process with reproduction mechanism $\Psi$ and immigration mechanism $\Phi$ is called CBI$(\Psi, \Phi)$ process. Kawazu and Watanabe \cite{KAW} establish that a CBI$(\Psi,\Phi)$ process is a Feller process with for generator the operator $L$ acting on $C^{2}(\mathbb{R}_{+})$ as follows \begin{multline}\label{generator1}
Lf(x):=\frac{\sigma^{2}}{2}xf''(x)+(b-\gamma x)f'(x)+x\int_{0}^{\infty}(f(x+z)-f(x)-z1_{[0,1]}(z)f'(x))\pi(\ddr z)\\
+\int_{0}^{\infty}\left(f(x+z)-f(x)\right)\nu(\ddr z).
\end{multline}
Apart if explicitly mentioned, to avoid the case of deterministic CBI processes, we shall always assume that one of the three conditions holds: $\sigma\neq0$,  $\pi\not \equiv 0$, or $\nu\not \equiv 0$. Moreover,
we assume that there exists $q\in\R_+$ such that $\Psi(q)>0$ (i.e. $-\Psi$ is not the Laplace exponent of a subordinator). This is equivalent to assume that the \textit{effective drift} \textbf{d} defined by
\begin{equation}\label{effdrift}
\textbf{d}:=\left\{ \begin{array}{l}
\gamma+\int_{0}^{1}z\pi(\ddr z) \mbox{ if the process has bounded variation}\\
\\
+\infty \mbox{ if the process has unbounded variation,}\\
\end{array}\right.
\end{equation}
belongs to $(0,\infty]$. Otherwise, the corresponding CBI process would be non-decreasing, and the problems studied in the present work are trivial. Notation $\mathbb{P}_{x}$ denotes the law of the process started at $x\in \mathbb{R}_+$, and $\mathbb{E}_{x}$ the corresponding expectation operator. Let $(X_{t}, t\geq 0)$ a CBI$(\Psi, \Phi)$ process, its one-dimensional marginal law satisfies:
\begin{equation}\label{laplace}
\mathbb{E}_{x}[e^{-qX_{t}}]=\exp\left(-xv_{t}(q)-\int_{0}^{t}\Phi(v_{s}(q))\ddr s\right),
\end{equation}
with $\frac{\partial v_{t}(q)}{\partial t}=-\Psi(v_{t}(q))$ and $v_{0}(q)=q$.\\

Recall the following classification (see Chapter 10 of \cite{Kyprianoubook} for details) : the branching mechanism $\Psi$ is said
\begin{itemize}
\item subcritical if $\Psi'(0+)>0$,
\item critical if $\Psi'(0+)=0$,
\item supercritical if $\Psi'(0+)<0$.
\end{itemize}
Throughout the article, we take the convention that for any finite real number $C$, $C/\infty=0$. We adopt the following definition of recurrence and transience.
\begin{Def} We say that the process $(X_{t}, t\geq 0)$ is recurrent if there exists an $x\in \mathbb{R}_{+}$ such that
\begin{equation} \label{recurrence} \mathbb{P}_{x}(\underset{t\rightarrow \infty}{\liminf}\ |X_{t}-x|=0)=1.
\end{equation}
On the other hand, we say that the process is transient if
\begin{equation} \label{transience} \mathbb{P}_{x}(\underset{t\rightarrow \infty}{\lim}X_{t}=\infty)=1 \text{ for every } x \in \mathbb{R}_{+}.
\end{equation}
\end{Def}
When the reproduction mechanism reduces to $\Psi(q)=\frac{\sigma^{2}}{2}q^{2}$ and $\Phi(q)=bq$, the process is the Feller diffusion, also called Cox-Ingersoll-Ross model in the financial setting. This is the unique solution to the stochastic equation :
\[X_{t}=x+\sigma\int_{0}^{t}\sqrt{X_{s}}\ddr B_{s}+bt,\]
where $(B_{t}, t\geq 0)$ is a Brownian motion. A standard method to study the hitting times, as well as the transience and recurrence of a general diffusion, is to use potential theory and scale functions (see for instance pages 128-129 of Itô and Mckean \cite{Ito}). This theory yields the following classic result concerning the Feller diffusion $(X_{t}, t\geq 0)$ : if $2b\geq \sigma^{2}$, then the point $0$ is polar. If $2b>\sigma^{2}$, the process is transient, otherwise the process is recurrent. In particular, if $2b=\sigma^{2}$, then $0$ is polar and the process is recurrent (we refer, for instance, to Chapter XI of Revuz-Yor \cite{RevuzYor} for a proof). \\

We shall study these path-properties for the general CBI processes. The polarity of zero has been studied in Foucart and Uribe Bravo \cite{ZeroCBI}. However, this latter work focuses on the zero-set and does not provide a criterion for transience or recurrence of the process. Moreover, as we shall see, zero may be polar and recurrent (in sense of (\ref{recurrence})). 
\\

Denote the first hitting time of the point $a$ by $\sigma_{a}$ :
\begin{equation}
\sigma_{a}:=\inf\{t>0; X_{t}=a\}.
\end{equation}

We highlight that the process has no downward jumps, therefore $\sigma_{a}$ is also the time of entrance in $\mathbb{R}_{+}\cap [0,a]$. We will discuss the law of $\sigma_{a}$ when the process starts from a state $x$ greater than $a$.\\

%

On the one hand, when the mechanism $\Psi$ reduces to $\Psi(q)=\gamma q$ with $\gamma>0$, the class of CBI processes corresponds to positive Ornstein-Uhlenbeck processes. This class of processes has been intensively studied. Hadjiev \cite{hadjiev1985first} get a formula for the hitting times of generalized Ornstein-Uhlenbeck processes. Patie \cite{patie2005martingale}, \cite{patiealea}, Novikov \cite{Novikov} apply potential theory to get identities for the joint law of $(\sigma_{a}, \int_{0}^{\sigma_{a}}X_{s}\ddr s)$, and for the first exit times. On the other hand, when no immigration is taken into account (namely, with $\Phi\equiv 0$), the corresponding CBI process is simply a continuous-state branching process (CB process) for which  many results have been obtained using the Lamperti transform (which relates any CB process to a spectrally positive Lévy process). We refer, for instance, to Chapter 10 of \cite{Kyprianoubook}. We mention that a Lamperti-type representation for the CBI processes has been obtained by Caballero et al. in \cite{caballero2013lamperti}. However, our methods do not rely on this representation. 
\\

Our main objective is to generalize some of these results when immigration is taken into account for a general reproduction mechanism $\Psi$. In this framework, the integral from $0$ to $\sigma_{a}$ of the process can be interpreted as the total population up to time $\sigma_a$. The results reveal the interplay between $\Phi$ and $\Psi$ in some path properties of CBI processes. The first main theorem is the following. Set $v=\frac{b}{\textbf{d}}$ with $\textbf{d}$ defined by (\ref{effdrift}).
\begin{thm}\label{mainth} Let $x> a\geq v$. For every $\lambda>0$ and $\mu\geq0$, we have
\begin{equation} \label{hitting time 2}
\mathbb{E}_{x}
\Big[\exp\Big\{-\lambda \sigma_{a}-\mu\int_0^{\sigma_a}X_t \ddr t\Big\}\Big]
=\frac{\int_{q(\mu)}^{\infty}\frac{\ddr z}{\Psi(z)-\mu}\exp\left(-xz+\int_{\theta}^{z}\frac{\Phi(u)+\lambda}{\Psi(u)-\mu}\ddr u\right)}{\int_{q(\mu)}^{\infty}\frac{\ddr z}{\Psi(z)-\mu}\exp\left(-az+\int_{\theta}^{z}\frac{\Phi(u)+\lambda}{\Psi(u)-\mu}\ddr u\right)},
\end{equation}
where $q(\mu):=\sup\{q\geq0: \Psi(q)=\mu\}$, and $\theta$ is an arbitrary constant larger than $q(\mu)$.
\end{thm}
When $\Phi$ is null or taken of the specific form $\Psi'$, some formulas are simplified and we recover certain results on continuous-state branching processes.
\\
\\
The second theorem yields a necessary and sufficient criterion for the recurrence or transience property of a CBI$(\Psi,\Phi)$  process when $\Phi\not \equiv 0$.
\begin{thm}\label{criterionrecurrence}
\begin{enumerate}
\item[(a)] In the critical or subcritical case, the process is recurrent or transient according as
\begin{equation}\label{crit}
\int_{0}^{1}\frac{\ddr z}{\Psi(z)}\exp \left[ -\int_{z}^{1}\frac{\Phi(x)} {\Psi(x)} \ddr x\right]=+ \infty \text{ or } < + \infty.
\end{equation}
\item[(b)] In the supercritical case, the CBI$(\Psi,\Phi)$ process is transient.
\end{enumerate}
\end{thm}

The paper is organized as follows. We begin by studying the state space of a CBI in Section \ref{state-space}. We then prove a key lemma (Section \ref{proofmain}) providing some $\lambda$-invariant functions, and apply it to  establish Theorem \ref{mainth}. We derive from Theorem \ref{mainth} a formula for the Laplace transform of the hitting times and get a criterion for the polarity of zero. In Section \ref{rectrans}, we establish firstly some direct corollaries of Theorem \ref{criterionrecurrence}. In particular,  we obtain the law of the minimum of a transient CBI. We then proceed to the proof of Theorem \ref{criterionrecurrence} and show how to construct null-recurrent CBIs. Eventually, we study the integral of the CBI process up to time~$\sigma_{a}$.

\section{State space of CBI processes.}\label{state-space}
We study here the state space of a general CBI process. A trivial example of CBI process which is not irreducible in $\mathbb{R}_{+}$ is the deterministic one. Namely, if $\Phi(q)=bq$  and $\Psi(q)=\gamma q$ with $\gamma>0$, the associated CBI is $X_{t}=X_{0}e^{-\gamma t}+\frac{b}{\gamma}(1-e^{-\gamma t}).$ The path of this process is above $\frac{b}{\gamma}$ as soon as $X_{0}>\frac{b}{\gamma}$.  
As already mentioned the case when $-\Psi$ is the Laplace exponent of a subordinator is excluded. Recall $\mathbf{d}>0$ and $v=\frac{b}{\textbf{d}}$.
We state a lower bound for any CBI process.
\begin{prop}
Let $X$ be a CBI$(\Psi,\Phi)$ process started at $x\in(0,\infty)$.
Then, $\P_x$ almost surely, for all $t>0$,
\begin{equation}\label{minoY}
X_t\geq e^{-\textbf{d}t}x+v\left(1-e^{-\textbf{d}t}\right).
\end{equation}
In particular, this implies $\underset{t\to\infty}{\liminf} \ X_t\geq v$.

\end{prop}
\begin{proof}
Firstly, one can notice that when $X$ has unbounded variation, then $\textbf{d}=\infty$ and $v=0$. The lower bound in the lemma is then null and the statement is clear. We then focus on the case of bounded variation and denote, for all $t>0$, $x_t:=e^{-\textbf{d}t}x+v\left(1-e^{-\textbf{d}t}\right)$. Using the càdlàg regularity, it will be sufficient to prove that for a fixed $t\in(0,\infty)$,
\begin{equation}\label{Yinfv}
\P_x\left(X_t<x_t\right)=0.
\end{equation}
Let $\tilde{X}$ be a CBI$(\Psi, \tilde{\Phi})$, where $\tilde{\Phi}(\lambda)=b\lambda$. We have then
for all $\lambda$,
$$\mathbb{E}_{x}[e^{-\lambda\tilde{X}_{t}}]=\exp\left(-xv_{t}(\lambda)-b\int_{0}^{t}v_{s}(\lambda)\ddr s\right),$$
thus
$\E_x[e^{-\lambda \tilde{X}_t}]\geq \E_x[e^{-\lambda X_t}]$, and therefore
$\P_x(X_t<x_t)\leq\P_x(\tilde{X}_t<x_t)$. We will show that the latter probability is $0$.

It is well-known that for a fixed $t>0$, the map $\lambda\mapsto v_t(\lambda)$ is the Laplace exponent of a subordinator (see for instance Bertoin-Le Gall \cite{LGB0}).
More precisely the underlying subordinator has for drift $e^{-\textbf{d}t}$ (see Duquesne and Labb\'e in \cite{DuqLab} Section 2.1 for details). Consider the Laplace exponent of the driftless subordinator :
$$w_t(\lambda):=v_t(\lambda)-e^{-\textbf{d}t}\lambda.$$
One can write
\begin{equation}
\mathbb{E}_{x}[e^{-\lambda\tilde{X}_{t}}]=\exp\left(-\lambda x_t -xw_{t}(\lambda)-b\int_0^t w_s(\lambda) \ddr s)\right),
\end{equation}
and
\begin{equation}
\mathbb{E}_{x}[\exp(-\lambda(\tilde{X_{t}}-x_{t})]=\exp\left(-xw_{t}(\lambda)-b\int_0^t w_s(\lambda) \ddr s)\right).
\end{equation}
One can plainly check that the map $\lambda\mapsto xw_{t}(\lambda)+b\int_0^t w_s(\lambda) \ddr s$ is the Laplace exponent of a non-negative random variable. We deduce $\tilde{X}_t\geq x_t$,  $\P_x$-a.s., and thus (\ref{Yinfv}).
\end{proof}
\begin{rem}
Alternatively, one can use stochastic calculus. Consider the case of bounded variation for which $\sigma=0$ and $\int_0^1 x\pi(\ddr x)<\infty$. Let $N_0(\ddr s,\ddr u)$ and $N_1(\ddr s,\ddr z,\ddr u)$ be two independent Poisson random measures on $(0,\infty)^2$ and $(0,\infty)^3$ with intensity $\ddr s\nu(\ddr z)$ and $\ddr s\pi(\ddr z)\ddr u$, respectively. For each $x\ge 0$ there is a pathwise unique positive strong solution to the following stochastic equation :
 \begin{eqnarray*}\label{CBI SDE}
X_t
 =
x + \int_0^t(b-{\rm\bf d} X_s) \ddr s  +\int_0^\infty z N_0(\ddr s,\ddr z)
+\int_0^t\int_0^\infty\int_0^{X_{s-}} z {N}_1(\ddr s, \ddr z,\ddr u).
 \end{eqnarray*}
By It\^{o}'s formula, the solution $(X_t,t\ge 0)$ is a CBI $(\Psi,\Phi)$ with $\sigma=0$; See Theorem~3.1 of Dawson and Li \cite{dawson2012stochastic}. On the other hand,
 \begin{eqnarray*}
x_t
 =
x + \int_0^t(b-{\rm\bf d} x_s) \ddr s.
 \end{eqnarray*}
It follows from Theorem~2.2 of Dawson and Li \cite{dawson2012stochastic} that $\mathbb{P}_x(X_t\geq x_t \mbox{ for all } t\geq0)=1$.
\end{rem}

In the (sub)critical case, a necessary and sufficient condition for the existence of a stationary distribution was announced by Pinsky \cite{Pinsky} and obtained by Li :
\begin{thm}[Theorem 3.20 in Li \cite{Li}]\label{positiverec}
\begin{itemize}
\item[i)] If $\int_{0}^{1}\frac{\Phi(u)}{\Psi(u)}\ddr u<\infty,$ then the CBI$(\Psi,\Phi)$ process, $(X_{t},t\geq 0)$,  has an invariant probability distribution. In the subcritical case ($\Psi'(0+)>0$), this integral condition is equivalent to
$$\int_{1}^{\infty}\log(u)\nu(\ddr u)<\infty.$$
\item[ii)] If $\int_0^1\frac{\Phi(u)}{\Psi(u)}\ddr u=\infty$, then for all $x,b\in\R_+$,
$$\lim\limits_{t\to\infty}\P_x(X_t\leq b)=0.$$
\end{itemize}
\end{thm}
\begin{rem}
The second statement of Theorem \ref{positiverec} is not plainly stated in \cite{Li}. Nevertheless, one can observe  in the proof of Theorem 3.20 in\cite{Li} that if $\int_0^1\frac{\Phi(u)}{\Psi(u)}\ddr u=\infty$, then  $\E_x\left[e^{-\lambda X_t}\right]\underset{t\to\infty}{\longrightarrow}0$. We refer also to the Appendix A of Keller-Ressel and Mijatovi\'c \cite{Keller}.
\end{rem}
It follows from Theorem \ref{positiverec} and Proposition 4.4 of \cite{Keller} that either $(X_t, t\geq 0)$ has the non-degenerate limit distribution with support $[v,\infty)$ or $X_t\overset{p}{\rightarrow}\infty$ as $t\rightarrow\infty$. Thus, applying Fatou's lemma, it is not hard to see that
 \beqlb\label{limsup}
 \P_x\Big(\limsup_{t\rightarrow\infty}X_t=\infty\Big)=1,\ \mbox{for any }x\in\R_+,
 \eeqlb
if $\Phi\not \equiv 0$. Starting from a point in $\mathcal{S}=[v,\infty)$, the process stays in $\mathcal{S}$, so we shall work with $\mathcal{S}$ as the state space.
Following the usual classification of Markov processes, a CBI process with a non-degenerate limit distribution is said to be positive recurrent. We shall see in the sequel that any positive recurrent process is indeed recurrent in the sense of (\ref{recurrence}).
\section{Proof of Theorem \ref{mainth}.}\label{proofmain}

Recall $L$ the infinitesimal generator of a CBI$(\Psi,\Phi)$ stated in (\ref{generator1}). Let $\mu\geq 0$ and set $\bpsi(q)=\Psi(q)-\mu$. Denote $\bar L$ the generator of $(\bar X_{t}, t\geq 0)$, a CBI$(\bpsi, \Phi)$. For all $f\in C^{2}(\mathbb{R}_{+})$ \[\bar L f(x)=Lf(x)-\mu xf(x).\] Recall $q(\mu)=\sup\{q\geq 0 : \Psi(q)=\mu\}$. Note that $q(\mu)<\infty$ since by assumption there exists $q$ such that $\Psi(q)>0$. 
We fix a constant $\theta=\theta(\mu)\in (q(\mu),\infty)$. The next Lemma provides some invariant functions for the generator $\bar L$.

\begin{lem}\label{key}
Let $\lambda,\mu\geq 0$. Define, for $x\in\left(q(\mu),\infty\right)$,
\begin{equation}\label{functiong}
g_{\lambda,\mu}(x):=\frac{1}{\Psi(x)-\mu}\exp\left[\int_{\theta}^{x}\frac{\Phi(u)+\lambda}{\Psi(u)-\mu}\ddr u\right],
\end{equation}
and $$f_{\lambda,\mu}(x):=\int_{q(\mu)}^{\infty}e^{-xz}g_{\lambda,\mu}(z)\ddr z.$$
If $\lambda>0$, the function $f_{\lambda,\mu}$ is a $C^{1}$-function decreasing on $(v,\infty)$ such that  \[\bar L f_{\lambda,\mu}=\lambda f_{\lambda,\mu}.\]  
\end{lem}

\begin{proof}[Proof of Lemma \ref{key}.]
Let $\lambda>0,\mu\geq 0$. Firstly, we check that $f_{\lambda,\mu}(x)$ is well-defined for $x>v$. We have
$$\frac{\Phi(u)}{\Psi(u)-\mu}=\frac{\Phi(u)}{u}\frac{u}{\Psi(u)-\mu} \underset{u\rightarrow +\infty}{\longrightarrow} \frac{b}{\textbf{d}}=:v,$$
therefore
$$\frac{1}{z}\int_{\theta}^{z}\frac{\Phi(u)+\lambda}{\Psi(u)-\mu}\ddr u\underset{z\to\infty}{\longrightarrow}v.$$
Since $x>v$ and $\Psi(z)-\mu\geq Cz$ with large enough $z$, and a constant $C>0$, we get for all $\lambda\geq 0$
\begin{equation}\label{integrability}
\int_{\theta}^{\infty}\frac{\ddr z}{\Psi(z)}\exp \left[-xz+\int_{\theta}^{z}\frac{\Phi(u)+\lambda}{\Psi(u)-\mu}\ddr u\right]<\infty.
\end{equation}
It remains to verify the integrability at $q(\mu)$. We have
$$\int_{q(\mu)}^{\theta}\frac{\ddr z}{\Psi(z)-\mu}\exp \left[-xz-\int_{z}^{\theta}\frac{\Phi(u)+\lambda}{\Psi(u)-\mu}\ddr u\right]\leq  \int_{q(\mu)}^{\theta}\frac{\ddr z}{\Psi(z)-\mu}\exp \left[-\int_{z}^{\theta}\frac{\lambda}{\Psi(u)-\mu}\ddr u\right].$$
Consider $\lambda>0$, an antiderivative of the integrand in the right hand side is
\begin{equation}\label{primitive}
z\mapsto \frac{1}{\lambda}\exp \left[-\lambda\int_{z}^{\theta}\frac{1}{\Psi(u)-\mu}\ddr u\right].
\end{equation}
This takes a finite value at $q(\mu)$ and yields the wished integrability. \\

\noindent Remark that $g_{\lambda,\mu}$ solves the ordinary differential equation
\begin{equation} \label{ODE} \Psi'(z)g_{\lambda,\mu}(z)+\left(\Psi(z)-\mu\right) g'_{\lambda,\mu}(z)=(\Phi(z)+\lambda)g_{\lambda,\mu}(z), \quad \forall z \in(q(\mu), \infty).
\end{equation}
For all $z$, define $h_{z}(x)=e^{-xz}$, one can easily check that
$$\bar{L}h_{z}(x)=\left[x(\Psi(z)-\mu)-\Phi(z)\right]h_{z}(x).$$
We compute
\begin{align*}
\bar{L} f_{\lambda,\mu}(x)-\lambda f_{\lambda, \mu}(x)&=\int_{0}^{\infty}\left(\bar{L}h_{z}(x)-\lambda h_{z}(x)\right)g_{\lambda,\mu}(z)dz\\
&=\int_{q(\mu)}^{\infty}e^{-xz}\left(x(\Psi(z)-\mu)-\Phi(z)-\lambda\right)g_{\lambda,\mu}(z)\ddr z\\
&=\int_{q(\mu)}^{\infty}e^{-xz}\left(\Psi'(z)g_{\lambda,\mu}(z)+(\Psi(z)-\mu)g'_{\lambda,\mu}(z)-(\Phi(z)+\lambda)g_{\lambda,\mu}(z)\right)\ddr z\\
&=0.
\end{align*}
The third equality follows from integration by parts. Indeed, we have $$(\Psi(x)-\mu)g_{\lambda,\mu}(x)=\exp\left(\int_{\theta}^{x}\frac{\Phi(u)+\lambda}{\Psi(u)-\mu}\ddr u\right)\underset{x\to 0}{\longrightarrow}0$$ because $\int_{q(\mu)+}\frac{\ddr u}{\Psi(u)-\mu}=\infty$, since $\Psi(u)-\mu$ is  always sub-linear near $q(\mu)$.
The last equality holds true because of the ODE (\ref{ODE}).
\end{proof}

We establish now Theorem \ref{mainth}.

\begin{proof}[Proof of Theorem \ref{mainth}]
Consider a CBI$(\Psi, \Phi)$ process $(X_{t}, t\geq 0)$ and define $I_t:=\int_0^t X_s\ddr s$. The family $(e^{-\mu I_{t}}, t\geq 0)$ is a continuous multiplicative functional of $(X_{t},t \geq 0)$. Denote the subordinate semi-group (in the terminology of Blumenthal and Getoor \cite{Blumenthal}) by $Q_{t}$, and the subprocess by $(\bar X_{t}, t\geq 0)$.  We have for all $f\in C^{2}(\mathbb{R}_{+})$
\beqnn
Q_tf(x)= \bar{\mathbb{E}}[f(\bar X_{t})]:=\mathbb{E}_x[f(X_t)e^{-\mu I_{t}}].
\eeqnn
We refer the reader to Theorem 3.3 and 3.12 pages 106 and 110 of Blumenthal and Getoor \cite{Blumenthal}. The bivariate process $\left((X_t,I_t); t\geq 0\right)$ is a Markov process. Similarly as Patie \cite{patieselfsimilar} (see Lemma 7), one can see by Itô's formula that for any function $f\in C_c^2(\mathbb{R}_+)$,
 \beqnn
 f(X_t)e^{-\mu\int_0^tX_s \ddr s}-f(x)-\int_0^t e^{-\mu\int_0^s X_u \ddr u}(Lf(X_s)-\mu X_sf(X_s))\ddr s
 \eeqnn
is a local martingale. Theorem 4.1.2 in \cite{Lilecturenotes} applies and ensures that $(\bar X_{t} ,t \geq 0)$ is a CBI$(\bpsi, \Phi)$ process. Firstly we consider $a>v$, and recall $\sigma_{a}=\inf\{t\geq 0, X_{t}=a\}$. From Lemma \ref{key}, one can apply Dynkin's formula to the Markov process $(\bar X_t, t\geq 0)$ killed at time $\sigma_a$, we get
$$\bar{\mathbb{E}}_{x}[e^{-\lambda \sigma_{a}\wedge t}f_{\lambda, \mu}(\bar X_{\sigma_{a} \wedge t})]=f_{\lambda, \mu}(x),$$ and thus
$$\mathbb{E}_{x}[e^{-\mu I_{\sigma_{a} \wedge t}}e^{-\lambda \sigma_{a} \wedge t }f_{\lambda, \mu}(X_{\sigma_{a}\wedge t})]=f_{\lambda, \mu}(x).$$
If we start from a point $x>a$, since the process has no downward jumps, $X_{t}>a$ for all time $t<\sigma_{a}$, and $f_{\lambda, \mu}(X_{t\wedge \sigma_{a}})\leq f_{\lambda, \mu}(a)$. Therefore the left hand side of the above equality is bounded and when $t\rightarrow \infty$, we get $$\mathbb{E}_{x}
\Big[\exp\Big\{-\mu\int_0^{\sigma_a}X_t\ddr t-\lambda \sigma_{a}\Big\}\Big]
=\frac{f_{\lambda, \mu}(x)}{f_{\lambda, \mu}(a)},$$
with the convention $e^{-\infty}=0$. To prove the formula in the case $a=v$, we notice that $\sigma_a$ is increasing towards $\sigma_v$, when $a\downarrow v$, by quasi-left continuity of the CBI. The result follows by monotonicity.
\end{proof}

\section{Hitting times and polarity of the boundary point.}
By a slight abuse of notation, define $f_{\lambda}:=f_{\lambda,0}$ and $g_{\lambda}:=g_{\lambda,0}$, that is to say
\begin{equation}\label{defglambda}
g_{\lambda}(x)=\frac{1}{\Psi(x)}\exp\left[\int_\theta^x\frac{\Phi(u)+\lambda}{\Psi(u)}\ddr u\right]
\end{equation}
and $f_{\lambda}(x)=\int_{q(0)}^\infty e^{-xz}g_{\lambda}(z)\ddr z$.
As a direct consequence of Theorem \ref{mainth}, when $\mu$ goes to $0$, we get the following corollary.
\begin{cor}\label{hittingtime}
For all $\lambda\in(0,\infty)$, and $x>a\geq v$
\begin{equation}\label{laplacesigma}
\E_x\left[e^{-\lambda\sigma_a}\right]=
\frac{f_{\lambda}(x)}{f_{\lambda}(a)}.
\end{equation}
\end{cor}
\begin{rem} We stress that the process $(e^{-\lambda t}f_{\lambda}(X_{t}), t\geq 0)$ is not a martingale. For instance applying the optional stopping theorem to the first-exit time $\tau_{b}:=\inf\{t>0, X_{t}>b\}$ yields a contradiction. In the same vein as scale functions for Lévy processes, one has to stop the process to get a martingale. This issue comes from the fact that $f_{\lambda}$ is not in the domain of the generator associated to the CBI$(\Psi, \Phi)$ process. Indeed, we can plainly check that for any mechanisms $\Psi, \Phi$: $\lvert f'_{\lambda}(0)\rvert=\infty$.
\end{rem}
To the best of our knowledge these functions do not appear in the literature even when no immigration is taken into account. Consider that particular case and assume here that $\Phi\equiv 0$. The CBI is then a CB$(\Psi)$ process.
In the supercritical case, an easy calculation of the limit when $\lambda$ goes to $0$ yields
$$\mathbb{P}_x\left(\sigma_a<\infty\right)
=\exp\left(-(x-a)q(0)\right),\quad \forall a\in]0,x].$$
Note that this equality holds for $a=0$ under the Grey's condition (see for instance Theorem 3.8 in \cite{Li}). Furthermore, the function $f_{\lambda}$ has a simpler expression. Indeed, since $\frac{z}{\Psi(z)}\rightarrow 1/{\textbf{d}}\in[0,\infty)$ as $z\rightarrow\infty$,
there exists $k>0$ such that for $x>0$,
\beqnn
e^{-xz}\exp\Big(\lambda\int_\theta^z\frac{\ddr u}{\Psi(u)}\Big)\leq e^{-xz}\exp\Big(\lambda k\int_\theta^z\frac{\ddr u}{u}\Big)
=e^{-xz}(z/\theta)^{\lambda k}\underset{z\rightarrow\infty} \longrightarrow0.
\eeqnn
Since $\Psi'(q(0))<\infty$, we also have $e^{-xz}\exp\Big(\lambda\int_\theta^z\frac{\ddr u}{\Psi(u)}\Big)
\rightarrow0$ as $z\rightarrow q(0)$. Integrating by parts, a notable cancellation occurs, we get :

$$f_{\lambda}(x)=\frac{x}{\lambda}\int_{q(0)}^{\infty}e^{-xz}\exp \left(\lambda \int_{\theta}^{z}\frac{\ddr u}{\Psi(u)}\right) \ddr z,$$
and then for $x>a>v$,
$$\mathbb{E}_{x}[e^{-\lambda \sigma_{a}}]=\frac{x\int_{q(0)}^{\infty}e^{-xz}\exp \left(\lambda \int_{\theta}^{z}\frac{\ddr u}{\Psi(u)}\right) \ddr z}{a\int_{q(0)}^{\infty}e^{-az}\exp \left(\lambda \int_{\theta}^{z}\frac{\ddr u}{\Psi(u)}\right) \ddr z}.$$
We return to the general case for which $\Phi \not \equiv 0$. When $v=0$, Corollary \ref{hittingtime} provides the Laplace transform of $\sigma_0$, the hitting time of $0$. We study now the polarity of the boundary. Recall that a point $a \in \mathcal{S}$ is said to be polar if for all $x\in \mathcal{S}$ such as $x\neq a$,
$$\mathbb{P}_{x}\left(\sigma_a<\infty\right)=0.$$

We recover and complete some results of \cite{ZeroCBI} through more classic techniques relying on Corollary \ref{hittingtime}. 
\begin{cor}\label{polar} The only point that may be polar is $v$. If $\textbf{d}<\infty$ then $v$ is polar. In the unbounded variation case, $v=0$ and $0$ is polar if and only if $$\int_{\theta}^{\infty}\frac{\ddr z}{\Psi(z)}\exp \left[\int_{\theta}^{z}\frac{\Phi(x)} {\Psi(x)} \ddr x\right]=\infty.$$
\end{cor}
\begin{rem} The integrability condition $\int_{\theta}^{\infty}\frac{1}{\Psi(z)}\ddr z<\infty$ implies that $\textbf{d}=\infty$, which entails that $v=0$. However, it is worth mentioning that none of these implications are equivalences.
\end{rem}
\begin{proof}[Proof of Corollary \ref{polar}]
Let $\lambda>0$. From Corollary \ref{hittingtime}, the point $a$ is polar if and only if $f_{\lambda}(a)=\infty$. We have seen that $f_{\lambda}(x)\in (0, \infty)$ for any $x\in (v, \infty)$. Thus only $v$ may be polar. Firstly, if $\textbf{d}<\infty$, note that $$\frac{\Phi(x)}{\Psi(x)}-v=\frac{1}{b\Psi(x)}\left[\textbf{d}\int_0^\infty(1-e^{-xu})\nu(\ddr u)+b\int_0^\infty(1-e^{-xu})\pi(\ddr u)\right]\geq0.$$ Then
\beqnn
\int_{\theta}^{\infty}\frac{\ddr z}{\Psi(z)}\exp \left[-vz+ \int_{\theta}^{z}\frac{\Phi(x)} {\Psi(x)} \ddr x\right]\geq e^{-v\theta}\int_\theta^\infty\frac{\ddr z}{\Psi(z)}=\infty,
\eeqnn
and therefore we have $f_{\lambda}(v)=\infty$.\\

Assume now $\textbf{d}=\infty$ (thus $v=0$) and
$\int_{\theta}^{\infty}\frac{\ddr z}{\Psi(z)}\exp \left[\int_{\theta}^{z}\frac{\Phi(x)} {\Psi(x)} \ddr x\right]=\infty$. We have $f_{\lambda}(0)=\infty$ and the same arguments hold. 
\\

We show now that if $\int_{\theta}^{\infty}\frac{\ddr z}{\Psi(z)}\exp \left[\int_{\theta}^{z}\frac{\Phi(x)} {\Psi(x)} \ddr x\right]<\infty,$ then
$$\mathbb{P}_{x}[\sigma_{0}<\infty]>0.$$
Writing
\begin{align}\label{equiv1}
\mathbb{E}_{x}[e^{-\lambda \sigma_{a}}]&=\frac{\int_{q(0)}^{\infty}e^{-xz}g_{\lambda}(z)\ddr z}{\int_{q(0)}^{\infty}e^{-az}g_{\lambda}(z)\ddr z} \nonumber \\
&=\frac{\int_{q(0)}^{\theta}e^{-xz}g_{\lambda}(z) \ddr z \left(1+\int_{\theta}^{\infty}e^{-xz}g_{\lambda}(z)\ddr z / \int_{q(0)}^{\theta}e^{-xz}g_{\lambda}(z)\ddr z\right)}{\int_{q(0)}^{\theta}e^{-az}g_{\lambda}(z) \ddr z \left(1+\int_{\theta}^{\infty}e^{-az}g_{\lambda}(z)\ddr z / \int_{q(0)}^{\theta}e^{-az}g_{\lambda}(z)\ddr z\right)}
\end{align}
for $a=0$, one can see that $\underset{\lambda \rightarrow 0}\lim \mathbb{E}_{x}[e^{-\lambda \sigma_{0}}]> 0$ since $$\frac{1}{\int_{q(0)}^{\theta}g_{\lambda}(z) \ddr z}\int_{\theta}^{\infty}g_{\lambda}(z)\ddr z \underset{\lambda \rightarrow 0}\longrightarrow \frac{1}{\int_{q(0)}^{\theta}g_{0}(z) \ddr z}\int_{\theta}^{\infty}g_{0}(z)\ddr z \in [0,\infty[.$$
\end{proof}
\section{Recurrence and transience.}\label{rectrans}
\subsection{Criterion of transience/recurrence and properties of transient CBIs.}
We restate Theorem \ref{criterionrecurrence} and provide some corollaries. We stress that in the (sub)critical case, $q(0)=0$ and we choose $\theta=1$.
\newpage
\noindent
\textbf{Theorem} \ref{criterionrecurrence}.
\begin{enumerate}
\item[(a)] In the critical or subcritical case, the process is recurrent or transient according as
\begin{equation*}
\int_{0}^{1}\frac{\ddr z}{\Psi(z)}\exp \left[ -\int_{z}^{1}\frac{\Phi(x)} {\Psi(x)} \ddr x\right]=+ \infty \text{ or } < + \infty.
\end{equation*}
\item[(b)] In the supercritical case, the CBI$(\Psi,\Phi)$ process is transient.\\
\end{enumerate}
\begin{rem}
\begin{itemize}
\item In light of Theorem \ref{positiverec}, when the mechanism $\Psi$ is (sub)critical and $\int_{0}^{1}\frac{\Phi(x)}{\Psi(x)}\ddr x<\infty$, then the CBI$(\Psi, \Phi)$ process is recurrent.
\item In the criterion, when the mechanism $\Psi$ is subcritical, one can replace $\Phi$ by the map $q\mapsto \int_{1}^{\infty}(1-e^{-qx})\nu(\ddr x)$. In other words, neither the continuous immigration nor its small jumps play a role for the process to be transient. Moreover, we should mention that when $\Psi(q)=\gamma q$, the criterion coincides with that of Shiga \cite{Shiga2}. Note that a subcritical CBI with $\Phi(q)=bq$ is always recurrent.
\item If the state $0$ is not polar, that is $$\int_{1}^{\infty}\frac{\ddr z}{\Psi(z)}\exp \left[ \int_{1}^{z}\frac{\Phi(x)} {\Psi(x)} \ddr x\right]<\infty$$
then one has the same necessary and sufficient conditions for both neighborhood-recurrence and point-recurrence (studied in \cite{ZeroCBI}) of the state $0$. Indeed, if $(X_{t}, t\geq 0)$ is recurrent, then $\int_{0}^{1}g_{0}(x)\ddr x=\infty$ and rewriting (\ref{equiv1}), we get $\P_x(\sigma_0<\infty)=1$ for every $x\in\R_+$. Since $\P_x(\limsup_{t\rightarrow\infty}X_t=\infty)=1$,  we have that $(X_t, t\geq 0)$ hit $0$ infinitely many times at arbitrary large times a.s.
\end{itemize}
 \end{rem}
\begin{Example}
Consider $\Psi(q)=dq^{\alpha}$, $\Phi(q)=d'q^{\beta}$ with $\alpha \in (1,2]$ and $\beta \in (0,1)$.
\begin{itemize}
\item If $\beta>\alpha-1$, the process is positive recurrent and $0$ is polar. 
\item If $\beta<\alpha-1$, the process is transient and $0$ is not polar.
\item If $\beta=\alpha-1$ and $\alpha \in (1,2)$, the process is recurrent if $d'/d\leq \alpha\!-\!1$ and transient if $d'/d>\alpha\!-\!1$. The point $0$ is polar if and only if $d'/d\geq \alpha\!-\!1$. We highlight that if $d'/d=\alpha\!-\!1$, $0$ is polar but $\underset{t\rightarrow \infty}{\liminf}\ X_{t}=0$. We point out that in this case, the CBI process is selfsimilar. Patie in \cite{patieselfsimilar} obtained the condition for $0$ to be polar via other arguments.
\end{itemize}
\end{Example}
Assume that the process $(X_{t}, t\geq 0)$ is transient. One can plainly check that the function
$$f_{0}(x)=\int_{q(0)}^{\infty}\frac{\ddr z}{\Psi(z)}\exp\left(-xz+\int_{\theta}^{z}\frac{\Phi(u)}{\Psi(u)}\ddr u\right)$$
takes finite values for all $x>v$. Applying Corollary \ref{hittingtime} and Theorem \ref{criterionrecurrence}, we obtain the following proposition.
\begin{prop}
Denote the overall infimum of the transient process $(X_{t}, t\geq 0)$ by $I$. We have
$$\mathbb{P}_{x}\left(I\leq a\right)=\mathbb{P}_{x}\left(\sigma_a<\infty\right)=\frac{f_{0}(x)}{f_{0}(a)}.$$
If $f_{0}(0)=\int_{q(0)}^{\infty}\frac{\ddr z}{\Psi(z)}\exp \left(\int_{\theta}^{z}\frac{\Phi(u)}{\Psi(u)}\ddr u \right)<\infty$ (i.e $0$ is not polar and the process is transient) then the law of $I$ has an atom at $0$.
\end{prop}
\begin{proof} Firstly, note that $\mathbb{P}_{x}[I\leq a]=\mathbb{P}_{x}[\sigma_{a}<\infty]$. By Theorem \ref{criterionrecurrence}, the integrability condition needed to define $f_{0}$ is satisfied. Taking $\lambda=0$, in the formula for the Laplace transform of $\sigma_a$, yields $\mathbb{P}_{x}[\sigma_{a}<\infty]=\frac{f_{0}(x)}{f_{0}(a)}.$
\\
\end{proof}
The CB$(\Psi)$ process conditioned to be non extinct is an important example of CBI process. As a direct corollary of Theorem \ref{criterionrecurrence}, we recover and complete some results due to Lambert (see Theorem 4.2-i in \cite{Lambert2}).
\begin{cor} The critical CB process conditioned to be non extinct is transient. Moreover, if the process starts at $x$, its minimum is uniformly distributed over $[0,x]$. The subcritical CB$(\Psi)$ process conditioned to be non extinct is recurrent or transient according to
$$\int_0^1\frac{\ddr z}{z}\exp\left(-\int_z^1  \left(\frac{1}{\Psi'(0+) u}-\frac{1}{\Psi(u)}\right)\ddr u\right)=+\infty \text{ or } <+\infty.$$
\end{cor}
\begin{proof}
Let $\Psi$ be a critical reproduction mechanism. Consider the case $\Phi=\Psi'$, the CBI$(\Psi,\Phi)$ process has the same law as the CB($\Psi$) process conditioned to the non extinction. In that case, we have clearly
$\int_{1}^{z}\frac{\Psi'(u)}{\Psi(u)}\ddr u=\log(\Psi(z))-\log(\Psi(1)),$ and therefore
$$\int_{0}^{1}\frac{\ddr z}{\Psi(z)}\exp\left(\int_{1}^{z}\frac{\Psi'(u)}{\Psi(u)}\ddr u\right)=\frac{1}{\Psi(1)}<\infty.$$
In order to deal with the minimum, one can readily check that $f_{0}(x)=1/x$. Thus, the random variable $I$ is uniformly distributed over $[0,x]$. For the subcritical case, plugging $\Phi=\Psi'-\Psi'(0+)$ in the integral of Theorem \ref{criterionrecurrence}, yields easily the statement.
\end{proof}
\begin{rem}
The fact that the minimum of a critical CBI$(\Psi, \Psi')$ is uniformly distributed can be obtained alternatively from Proposition 3 in Chaumont \cite{Chaumont}, which states the corresponding result for Lévy processes conditioned to stay positive.  
Indeed, Lambert, in \cite{Lambert}, shows that the CB process conditioned to be non-extinct has the same law as a time-changed Lévy process conditioned to stay positive. 
\end{rem}

\subsection{Proof of Theorem \ref{criterionrecurrence}.}
Firstly, we establish statement $(a)$. The proof relies on the study of the Laplace transform of the hitting times provided by Corollary \ref{hittingtime}.
Recall $$g_{\lambda}(x)=\frac{1}{\Psi(x)}\exp\left[ \int_{1}^{x}\frac{\Phi(u)+\lambda}{\Psi(u)} \ddr u \right].$$ Equation (\ref{integrability}) ensures that for $x>v$,
$$\int_{1}^{\infty}e^{-xz}g_{0}(z)\ddr z<\infty.$$


\noi\textit{\textbf{Recurrence}.} Assume that $$\int_{0}^{1}\frac{1}{\Psi(x)}\exp\left[ \int_{1}^{x}\frac{\Phi(u)}{\Psi(u)} \ddr u \right]\ddr x= \infty.$$
For every $x\geq a$, $$\mathbb{P}_{x}[\sigma_{a}<\infty]=\underset{\lambda \rightarrow 0}\lim \mathbb{E}_{x}[e^{-\lambda \sigma_{a}}].$$
 Rewriting Equation (\ref{equiv1}), we have for all $a>v$,
\begin{align*}
\mathbb{E}_{x}[e^{-\lambda \sigma_{a}}] &=\frac{\int_{0}^{1}e^{-xz}g_{\lambda}(z) \ddr z \left(1+\int_{1}^{\infty}e^{-xz}g_{\lambda}(z)\ddr z / \int_{0}^{1}e^{-xz}g_{\lambda}(z)\ddr z\right)}{\int_{0}^{1}e^{-az}g_{\lambda}(z) \ddr z \left(1+\int_{1}^{\infty}e^{-az}g_{\lambda}(z)\ddr z / \int_{0}^{1}e^{-az}g_{\lambda}(z)\ddr z\right)}\\
&\underset{\lambda \rightarrow 0}{\longrightarrow} 1. \nonumber
\end{align*}
We deduce that $\P_x(\sigma_a<\infty)=1$ for any $x\geq a>v$, which implies
$\P_x\left(\liminf\limits_{t\to\infty}X_t\leq  v\right)=1$.
The lower bound of Lemma \ref{minoY} then entails $\P_x\left(\liminf\limits_{t\to\infty}X_t= v\right)=1$, so the process is recurrent in sense of (\ref{recurrence}).
\\
\\
\noi\textit{\textbf{Transience}.}
We now work under the assumption
\begin{equation}\label{int0}
\int_{0}^{1}\frac{\ddr z}{\Psi(z)}\exp\left(-\int_{z}^{1}\frac{\Phi(u)}{\Psi(u)}\ddr u\right)<\infty.
\end{equation}
Let $a>v:=b/\textbf{d}$. We show that $\P_x\left(\liminf\limits_{t\to\infty}X_t< a\right)=0$.
One has
\begin{align}
\P_x\left(\liminf\limits_{t\to\infty}X_t< a\right)
&\leq\lim\limits_{t\to\infty}\P_x\left(\sigma_a\circ\theta_t<\infty\right)\\
&=\lim\limits_{t\to\infty}\E_x\left[\P_{X_t}\left(\sigma_a<\infty\right)\right].\nonumber
\end{align}
Moreover one can write that,
\begin{equation}\label{decompo}
\E_x\left[\P_{X_t}\left(\sigma_a<\infty\right)\right]\leq \P_x\left(X_t\leq a\right)+\E_x\left[\un_{\{X_t>a\}}\P_{X_t}\left(\sigma_a<\infty\right)\right].
\end{equation}
Firstly, under (\ref{int0}), one has $\int_{0}^{1}\frac{\Phi(u)}{\Psi(u)}\ddr u=\infty$.  According to  ii) in Theorem \ref{positiverec}, it implies that
$$\lim\limits_{t\to 0}\P_x\left(X_t\leq a\right)=0.$$

Thus, the first term in (\ref{decompo}) goes to $0$ when $t\to\infty$. We focus now on the second term. Under (\ref{int0}), one can take $\lambda=0$ in Corollary \ref{hittingtime}. For $x>a>v$,
\begin{equation}\label{hitting time 3}
\P_x\left(\sigma_a<\infty\right)=\frac{\int_{0}^{\infty}\frac{\ddr z}{\Psi(z)}\exp\left(-xz+\int_{1}^{z}\frac{\Phi(u)}{\Psi(u)}\ddr u\right)}{\int_{0}^{\infty}\frac{\ddr z}{\Psi(z)}\exp\left(-az+\int_{1}^{z}\frac{\Phi(u)}{\Psi(u)}\ddr u\right)}=c_{a}\int_{0}^{\infty}g_{0}(z)e^{-xz}\ddr z.
\end{equation}

Hence,
\begin{align}\label{secondterme}
\E_x\left[\un_{\{X_t>a\}}\P_{X_t}\left(\sigma_a<\infty\right)\right]
&=c_a\E_x\left[\un_{\{X_t>a\}}\int_0^\infty g_{0}(z)e^{-zX_t}\ddr z\right]\nonumber\\
&=c_a\int_0^\infty g_{0}(z) \E_x\left[\un_{\{X_t>a\}}e^{-zX_t}\right]\ddr z.
\end{align}
Moreover, $\E_x\left[\un_{\{X_t>a\}}e^{-zX_t}\right]\leq e^{-za}$ and by (\ref{int0}) and (\ref{integrability}), $\int_{0}^{\infty}g_{0}(z)e^{-za} \ddr z <\infty$.
Furthermore, since $\int_{0+}\frac{\Phi(u)}{\Psi(u)}\ddr u=\infty$,
\begin{equation}\label{domconv}
\E_x\left[\un_{\{X_t>a\}}e^{-zX_t}\right]\leq \E_x\left[e^{-zX_t}\right]=\exp\left(-xv_t(z)-\int_{v_t(z)}^z\frac{\Phi(u)}{\Psi(u)}\ddr u\right)\underset{t\rightarrow \infty}\longrightarrow 0.
\end{equation}
Thus, by dominated convergence, the integral (\ref{secondterme}) tends to $0$, which entails the desired result.
Therefore the process is transient in the sense of Definition (\ref{transience}).
\\


In order to prove statement $(b)$  (transience in the supercritical case), one has just to adapt the proof above. Indeed, we have $\int_{q(0)}^\theta\frac{\ddr z}{\Psi(z)}\exp\left(-\int_z^\theta\frac{\Phi(u)}{\Psi(u)}\ddr u\right)<\infty$, so we can write (\ref{hitting time 3}). Moreover, one can use that  $\int_{v_t(z)}^z\frac{\Phi(u)}{\Psi(u)}\ddr u\underset{t\to\infty}{\longrightarrow}\int_{q(0)}^z\frac{\Phi(u)}{\Psi(u)}\ddr u=\infty$ in (\ref{domconv}).
\qed
\subsection{Construction of subcritical null-recurrent CBI processes.}
We look here for examples of null recurrent CBI processes. Assume that $\Psi(q)=\gamma q$, with $\gamma>0$. The following computations remains valid if $\Psi$ is subcritical with $\Psi'(0+)>0$, because only the behaviour of $\Psi$ at $0$ matters. To avoid positive recurrence, we need to choose $\Phi$ such that $\int_0\frac{\Phi(q)}{q}\ddr q=\infty$, which is equivalent to
\begin{equation}\label{norecpositive}
\int^\infty\log(u)\nu(\ddr u)=\infty.
\end{equation}
Moreover, to get a recurrent process, we know from Theorem \ref{criterionrecurrence} that $\Phi$ has to satisfy
\begin{equation}\label{recps}
\int_0^1\frac{\ddr z}{z}\exp\left(-\int_z^1\frac{\Phi(u)}{\gamma u}\ddr u\right)=\infty.
\end{equation}
From condition (\ref{norecpositive}), we know that the example of $\Phi$ we are looking for is not a deterministic drift. Moreover, when $\nu$ is not null, the value of the drift coefficient $b$ has no influence for (\ref{recps}) to be fulfilled. Therefore, we will take $b=0$ and we will exhibit a sufficient condition involving the Lévy measure $\nu$ to get (\ref{recps}). Denote $\onu(u):=\nu\left([u,\infty)\right)$ and  recall from Chapter III of Bertoin \cite{Ber96} that there exists a universal constant $\kappa$ such that $$\Phi(q)/q\leq \kappa J_\Phi(1/q), \forall q>0,\quad  {\rm where\ } J_\Phi(x):=\int_0^x \onu(u)\ddr u, x>0.$$
Thus, we have
\begin{align*}
\int_z^1\frac{\Phi(u)}{\gamma u}\ddr u\leq \frac{\kappa}{\gamma}\int_z^1 J_\Phi(1/u)\ddr u&=\frac{\kappa}{\gamma}\int_1^{1/z} J_\Phi(u)/u^2\ddr u\\
&=\frac{\kappa}{\gamma}\left(J_\Phi(1)-zJ_\Phi(1/z)+\int_1^{1/z}\onu(u)/u\ddr u\right),
\end{align*}
by integration by parts. Hence, a sufficient condition to get (\ref{recps}) is

\begin{equation}\label{recsuff}
\int_0^1\frac{\ddr z}{z}\exp\left(-\frac{\kappa}{\gamma}\int_{1}^{1/z}\onu(u)/u\ddr u\right)=\infty.
\end{equation}

\noi \textbf{Example 1.}
We consider $\alpha\in\R$ and define  $\nu$ such that
\begin{equation}\label{contreex1}
\int_1^{1/z}\onu(u)/u\ddr u=\alpha\log\log 1/z\quad {\rm up\ to\ an\ add.\ constant.},
\end{equation} so that the integral in (\ref{recsuff}) is of the same nature as $\int_0 \frac{\ddr z}{z\log(1/z)^{\kappa\alpha/\gamma}}$. The integral will be infinite if $\alpha$ is chosen such that $\kappa\alpha/\gamma\leq 1$. We can get (\ref{contreex1}) taking $\onu(u):=\alpha u\frac{\ddr}{\ddr u}\log\log u=\frac{\alpha}{\log u}$ on $[100,\infty]$,  that is $\nu(\ddr u)=\frac{\alpha}{u\log^2 u}\un_{[100,\infty]}\ddr u$. We can easily check that $\nu$ is a Lévy measure and that the condition (\ref{norecpositive}) is satisfied. This example is related to that given by Sato and Yamazato in Section 7 of \cite{sato1984operator}, in which  the authors highlight as remarkable that the null recurrence or transience of the process is function of $\kappa/\gamma$. The form of the criterion (\ref{recps}) and the rôle played by Bertrand's integrals provide a better understanding of the criterion. In the next example, the value of $\gamma$ has no influence.

\vskip 0.5cm
\noi \textbf{Example 2.}
We choose $\nu$ such that
\begin{equation}\label{contreex2}
\int_1^{1/z}\onu(u)/u\ddr u=\log\log\log 1/z \quad {\rm up\ to\ an\ add.\ constant.},
\end{equation} so that the integral in (\ref{recps}) is  $\int_0 \frac{\ddr z}{z\log\log(1/z)^{\kappa/\gamma}}=\infty$. We can get (\ref{contreex2}) taking
$$\onu(u):=u\frac{\ddr}{\ddr u}\log\log\log u=\frac{1}{\log u\log\log u}, \text{ on } [100,\infty],$$
that is $\nu(\ddr u)=\frac{\log\log(u)+1}{u\log^2 u\log^2(\log u)}\un_{[100,\infty]}\ddr u$. The density of the last Lévy measure is equivalent at $\infty$ to $\frac{1}{u\log^2 u\log\log u}$. Hence, we can check that $\nu$ is indeed a Lévy measure and that it satisfies (\ref{norecpositive}).

\section{Total population.}

As already said, one can see the integral $\int_0^{\sigma_a}X_s\ddr s$ as the total population up to time $\sigma_a$. In the case of the CB$(\Psi)$ ($\Phi\equiv 0$), and $a=0$, this is known as the total progeny. The corresponding integral $\int_0^{t}X_s\ddr s$ happens to be the time change in the Lamperti transform relating a CB$(\Psi)$ process with a spectrally positive Lévy process of Laplace exponent $\Psi$. This allows ones to transfer the study of $\int_0^{\sigma_a}X_s\ddr s$ to that of the hitting time of a Lévy process. See Bingham \cite{Bingham} and Corollary 10.9 in Kyprianou \cite{Kyprianoubook}. In what follows, we recover the latter corollary and obtain its analogue with immigration.
\begin{prop}
\label{totalproj}
Let $x>a\geq v$, and assume that $\Phi\equiv 0$. For all $\mu>0$,
\begin{equation}\label{laplacetotalproj}
\mathbb{E}_{x}
\Big[\exp\Big\{-\mu\int_0^{\sigma_a}X_t\ddr  t\Big\}\Big]
=\exp\left(-(x-a)q(\mu)\right).
\end{equation}
\end{prop}
\begin{proof}
Firstly, let $\lambda>0$. Integrating by parts, we have
$$\int_{q(\mu)}^{\infty}\frac{\ddr z}{\Psi(z)-\mu}\exp\left(-xz+\int_{\theta}^{z}\frac{\lambda}{\Psi(u)-\mu}\ddr u\right)=\int_{q(\mu)}^{\infty}\ddr z \ xe^{-xz}\exp\left(-xz+\int_{\theta}^{z}\frac{\lambda}{\Psi(u)-\mu}\ddr u\right),$$
which tends to $\int_{q(\mu)}^{\infty}\ddr z \ xe^{-xz}=\exp\left(-x q(\mu)\right)$, as $\lambda$ goes to $0$. Thus, let $a>v\geq 0$. The desired result follows from Theorem \ref{mainth}, with $\Phi\equiv 0$, letting $\lambda\to 0$. One can obtain now the case $a=v$ by monotonicity and quasi-left continuity.
\end{proof}
More generally, we have the following corollary of Theorem \ref{mainth}.
\begin{cor}\label{totalpop}
Let $x>a\geq v$, and assume that $\Phi\not\equiv 0$. For all $\mu>0$,
\begin{equation}\label{laplacetotalpop}
\mathbb{E}_{x}
\Big[\exp\Big\{-\mu\int_0^{\sigma_a}X_t \ddr t\Big\}\Big]
=\frac{\int_{q(\mu)}^{\infty}\frac{\ddr z}{\Psi(z)-\mu}\exp\left(-xz+\int_{\theta}^{z}\frac{\Phi(u)}{\Psi(u)-\mu}\ddr u\right)}{\int_{q(\mu)}^{\infty}\frac{\ddr z}{\Psi(z)-\mu}\exp\left(-az+\int_{\theta}^{z}\frac{\Phi(u)}{\Psi(u)-\mu}\ddr u\right)}.
\end{equation}
In the particular case of the CBI$(\Psi,\Psi')$ with $\Psi'(0)=0$ (this is the CB$(\Psi)$ conditioned to be non extinct), we have
$$\mathbb{E}_{x}
\Big[\exp\Big\{-\mu\int_0^{\sigma_a}X_t \ddr t\Big\}\Big]
=\frac{a}{x}\exp\left(-(x-a)q(\mu)\right),\quad \forall \mu>0, x>a\geq v.$$
\end{cor}
\begin{proof}
It follows readily from Theorem \ref{mainth} by letting $\lambda\to 0$. We only have to check that the integral in the numerator of (\ref{laplacetotalpop}) is finite. At infinity, this follows from (\ref{integrability}). At $q(\mu)$, one can use that
$$\Psi'(z)-\mu \underset{z\to q(\mu)}{\sim}\Psi'(q(\mu))\left(z-q(\mu)\right), \quad {\rm and } \quad  \int_{\theta}^{z}\frac{\Phi(u)}{\Psi(u)-\mu}du\underset{z\to q(\mu)}{\sim}\frac{\Phi\left(q(\mu)\right)}{\Psi'\left(q(\mu)\right)}
\log\left(\frac{z-q(\mu)}{\theta-q(\mu)}\right),$$
where $\Phi\left(q(\mu)\right)$ and $\Psi'\left(q(\mu)\right)\in(0,\infty)$ because $\mu\in(0,\infty)$.
\end{proof}

\textbf{Acknowledgement.} C.F thanks Juan Carlos Pardo for pointing several references, Alex Mijatovi\'c and Martin Keller-Ressel for fruitful discussions. C.M. is supported by NSFC (No.11001137, 11271204).

\end{document}